\pdfoutput=1
\documentclass[11pt]{amsart}
\usepackage{amssymb}
\usepackage{multirow}
\usepackage{mathtools}
\usepackage{mathrsfs}

\usepackage{hyperref}

\usepackage{tikz}
\usepackage{tikz-cd}
\usepackage{adjustbox}

\usepackage{graphicx}
\usepackage{color}
\usepackage{cleveref}

\usepackage{thmtools, thm-restate}
\usepackage[margin=1in]{geometry}

\newtheorem{thm}{Theorem}[section]
\newtheorem{cor}[thm]{Corollary}
\newtheorem{prop}[thm]{Proposition}
\newtheorem{lemma}[thm]{Lemma}

\theoremstyle{definition}
\newtheorem{defn}[thm]{Definition}

\theoremstyle{remark}

\newcommand{\Q}{\mathbb{Q}}
\newcommand{\Z}{\mathbb{Z}}
\newcommand{\C}{\mathbb{C}}

\DeclareMathOperator{\Sol}{\mathcal{S}\mathit{ol}}
\DeclareMathOperator{\DR}{\mathcal{DR}}
\DeclareMathOperator{\cHom}{\mathcal{H}\mathit{om}}

\DeclareMathOperator{\gr}{gr}

\DeclareMathOperator{\cone}{Cone}
\DeclareMathOperator{\Perv}{Perv}
\DeclareMathOperator{\For}{For}

\newcommand{\cA}{\mathcal{A}}
\newcommand{\cB}{\mathcal{B}}
\newcommand{\cD}{\mathcal{D}}

\newcommand{\cM}{\mathcal{M}}

\newcommand{\cN}{\mathcal{N}}
\newcommand{\cO}{\mathcal{O}}

\begin{document}
\title{Period Integrals and Hodge Modules}
\author{Laure Flapan}
\address{Department of Mathematics, Massachusetts Institute of Technology, Cambridge, MA 02142}
\email{lflapan@mit.edu}

\author{Robin Walters}
\address{Department of Mathematics, Northeastern University, Boston, MA 02115} \email{r.walters@northeastern.edu}

\author{Xiaolei Zhao}
\address{Department of Mathematics, University of California Santa Barbara, Santa Barbara, CA 93106}
\email{xlzhao@ucsb.edu}


\begin{abstract}
We define a map $\mathcal{P}_M$ attached to any polarized Hodge module $M$ such that the restriction of $\mathcal{P}_M$ to a locus on which $M$ is a variation of Hodge structures  induces the usual period integral pairing for this variation of Hodge structures. In the case that $M$ is the minimal extension of a simple polarized variation of Hodge structures $V$, we show that the homotopy image of $\mathcal{P}_M$ is the minimal extension of the graph morphism of the usual period integral map for $V$.
\end{abstract}
\maketitle

\section{Introduction}
Classically, the period integrals of a family $f\colon \mathscr{X}\rightarrow B$ of smooth projective $m$-dimensional varieties are obtained by choosing a symplectic basis $\gamma_1, \ldots, \gamma_r$ of $H_m(X_b,\Z)$ valid in some open neighborhood $U\subset B$ of $b\in B$ and then computing the values
$\int_{\gamma_i}\omega_{b}$
where $\omega_b\in H^0(X_b,\Omega_{X_b}^1)$. 

In the case that $X_b$ is a smooth algebraic curve, then, via the Torelli theorem, the curve $X_b$ is exactly determined by the values of its period integrals. More generally, if $V$ is a variation of Hodge structures of weight $w$ on a complex algebraic variety $B$ one can, as above, define a period pairing
\begin{equation}\label{eq: VHS pairing defn}
V_\Q\otimes F^wV_\C\rightarrow \mathcal{O}_B, \hspace{.2in}
\langle \gamma, \omega\rangle \mapsto \int_{\gamma}\omega.
\end{equation}

The asymptotic behavior of the periods as a smooth variety degenerates can sometimes contain significant information about the degeneration itself and the resulting singular variety. For instance in the case of a family of smooth algebraic curves $X_b$ degenerating to a singular fiber $X_0$, the limit of the periods of the $X_b$ as $b$ approaches $0$ are the integrals of the normalization of $X_0$. In fact the singular fiber $X_0$ is completely determined by these limit periods \cite[pg 35--36]{carlson}. However, classically there is not a good notion of the periods of singular varieties, making it hard to study periods of families of projective varieties where not all fibers are smooth. 

This stems from the fact that classical Hodge theory, and hence the classical study of periods, provides a framework within which to study either the cohomology of a family of smooth projective varieties, which it endows with a variation of Hodge structure, or the cohomology of a single complex algebraic variety, which it endows with a mixed Hodge structure. Saito's theory of Hodge modules, introduced in \cite{saito} and \cite{saito2}, provides a framework to study the cohomology of arbitrary families of complex algebraic varieties, in particular without requiring smoothness of the fibers, by uniting Hodge theory with the theories of perverse sheaves and $\mathcal{D}$-modules. For an introduction to Saito's theory, see for instance \cite{schnellnotes}.

The goal of this paper is to  formalize a notion of period integrals for (pure) Hodge modules.
\begin{defn}
 If $M$ is a polarized pure Hodge module of weight $w$ on a complex algebraic variety $B$ of dimension $d$, the 
 \textit{period integral map} $\mathcal{P}_M$ for the polarized Hodge module $M$ is a map (constructed in Definition \ref{def: PI map HM})
 \[\mathcal{P}_M\colon K\rightarrow R\cHom_{\cO_B}(\omega_B^{-1}\otimes_{\mathcal{O}_B}F_{-w}\mathcal{M},\cO_B)[d],\]
 where $K$ is the rational perverse sheaf underlying $M$ and $\mathcal{M}$ is the underlying regular holonomic $\mathcal{D}$-module equipped with filtration $F_\bullet$.
\end{defn}
 From this period integral map $\mathcal{P}_M$ we construct an isomorphism of perverse sheaves 
\[\gamma_M\colon K\rightarrow L[-1]\]
 through which $\mathcal{P}_M$ factors, where the perverse sheaf $L$ is the homotopy image of $\mathcal{P}_M$, meaning that $L$ is the mapping cone of the natural map \[R\cHom_{\cO_B}(\omega_B^{-1}\otimes_{\mathcal{O}_B}F_{-w}\mathcal{M},\cO_B)[d]\rightarrow \cone(\mathcal{P}_M).\]

In order to see that $\mathcal{P}_M$ together with $\gamma_M$ provide the correct generalization of the period pairing for a polarized variation of Hodge structures $V$, we consider the pairing \eqref{eq: VHS pairing defn} as a morphism
\[\mathbb{P}_{V}\colon \mathbb{V}\rightarrow \cHom_{\mathcal{O}_B}(F^w\mathcal{V},\mathcal{O}_B),\]
where $\mathbb{V}$ is the local system of $\Q$-vector spaces on $B$ underlying $V$ and $\mathcal{V}=\mathbb{V}\otimes_\Q \mathcal{O}_B$ is the corresponding vector bundle equipped with the Hodge filtration $F^\bullet$. Then since any polarized variation of Hodge structures $V$ corresponds to a polarized Hodge module  $M_V$, we
 prove 

\begin{restatable*}{prop}{propVHS}\label{prop: VHS}
If  $M=M_V$ is a polarized Hodge module attached to some polarized variation of Hodge structures $V$, then $\mathcal{P}_M=\mathbb{P}_{V}[d]$.
\end{restatable*}

Recall that, as proven in \cite{saito}, any polarized variation of Hodge structures $V$ on an open subset $U\subset B$ of a complex algebraic variety extends uniquely to a polarized Hodge module $M$ on $B$ via the minimal extension functor $j_{!*}$. This minimal extension functor for Hodge modules is compatible with the minimal extension functors for the underlying perverse sheaves and $\mathcal{D}$-modules. 

We thus use Proposition \ref{prop: VHS} in order to prove 

\begin{restatable*}{thm}{thmMinExt}\label{prop: min ext}
Suppose that $M\in HM(B,w)$ is a polarized Hodge module on a $d$-dimensional complex algebraic variety $B$ such that $M=j_{!*}V$ for some simple polarized variation of Hodge structures $V$ on an open subset $j\colon U\hookrightarrow B$. Then the isomorphism $\gamma_M$ satisfies
\[\gamma_M=j_{!*}(\Gamma_{\mathbb{P}_V[d]}),\]
where $\Gamma_{\mathbb{P}_V[d]}[d]\colon \mathbb{V}[d]\rightarrow \mathbb{V}\oplus \cHom_{\mathcal{O}_U}(F^w\mathcal{V},\mathcal{O}_U)[d]$ is the graph morphism of the morphism $\mathbb{P}_V[d]$.
\end{restatable*}

We moreover study some of the properties of the Hodge module period integral map $\mathcal{P}_M$. We prove in Propositions \ref{prop: inclusion} and  \ref{prop: M decomp} that $\mathcal{P}_M$ behaves well with respect to  polarized morphisms and polarized direct sums respectively and consequently in Corollaries \ref{cor: gamma incl} and \ref{cor: gamma decomp} respectively that the same is true for $\gamma_M$. 

One of Saito's fundamental results about Hodge modules is the Structure Theorem \cite{saito2}, which asserts that 
any polarized Hodge module can be decomposed as a direct sum of minimal extensions of variations of Hodge structures. More precisely, if $M$ is a polarized pure Hodge module on a complex algebraic variety $B$, then $B$ admits a stratification into locally closed subvarieties $B_i$ such that 
\begin{equation}\label{eq: M structure intro}M=\bigoplus_i (j_{i})_{!*}V_{B_i},\end{equation}
where $V_{B_i}$ is a simple variation of Hodge structures on $B_i$ and $j_i\colon B_i\hookrightarrow \overline{B}_i$ is the natural inclusion of $B_i$ into its closure. Thus we deduce 

\begin{restatable*}{cor}{corStructure}\label{cor: structure}
Let $B$ be a complex algebraic variety and $M\in HM(B,w)$ a polarizable Hodge module on $B$ equipped with a decomposition of the form \eqref{eq: M structure intro} for some stratification $B=\bigsqcup B_i$, where $\dim B_i=d_i$. Then, endowing $M$ with the polarization $\phi_M=\bigoplus_i\phi_{M_i}$, where each $\phi_{M_i}$ is a chosen polarization of the Hodge module  $M_i=(j_{i})_{!*}V_{B_i}$, we have
\[\gamma_M=\bigoplus_i (j_{i})_{!*}\Gamma_{\mathbb{P}_{V_{B_i}}[d_i]}.\]
\end{restatable*}

 Lastly, we examine the period integral map $\mathcal{P}_M$ in some specific cases where one has additional information about the weight $w$ Hodge module $M$ and, more importantly, its filtered piece $F_{-w}\mathcal{M}$. 

\subsection{The case of a one-dimensional base} The first such instance we explore is the case when the base $B$ for the Hodge module has dimension $1$ and $M$ is the minimal extension of a variation of Hodge structures $V$ on an open subset $U\subset B$. In this case, the nearby cycles functor $\Psi_f M$ applied to $M$, where $f\colon B \rightarrow \C$ is some non-constant holomorphic function, coincides with Deligne's limit mixed Hodge structure $L_fV$ at a point $P=f^{-1}(0)$. Hence, as we formalize in  Proposition \ref{prop: curve}, the period integral maps on the weight filtration $W(N_u)$-graded pieces of $\Psi_f M$ and $L_fV$ coincide. 

\subsubsection{One-dimensional families of curves} In the case that the variation of Hodge structures $V$ arises from the cohomology of a family $\mathscr{X}\rightarrow B$ of smooth algebraic curves, we then deduce in Section \ref{sec: curves} that 
\[\mathcal{P}_{\gr^{W(N_u)}_{1}\Psi_fM[-1]}=\bigoplus_i\mathbb{P}_{H^1(X_i)},\]
where the $X_i$ are the irreducible components of the normalization of the singular curve above the point $P\in B$. 

\subsubsection{Mildly singular degenerations} Additionally, in the case of a family of smooth varieties on a one-dimensional base with only mildly singular degenerations, recent results of Kerr--Laza \cite{KL} relate the first graded piece of the Hodge filtration on the limit mixed Hodge structure to the cohomology of the singular fibers. Hence under the hypotheses of Kerr--Laza \cite[Theorem 1.2]{KL}, we combine Proposition \ref{prop: curve} with the fact, proved in Proposition \ref{prop: inclusion}, that $\mathcal{P}_M$ behaves well with respect to polarized morphisms to 
 conclude in 
 Proposition \ref{prop: kerr-laza} that 
\[\mathcal{P}_{\gr^{W(N_u)}_{\ell-w+1}i^*M}=
\mathcal{P}_{\gr^{W(N_u)}_{\ell-w+1}\Psi_fM[-1]}\]
for all $\ell \in \mathbb{Z}$, where $i\colon P\hookrightarrow B$ is the natural inclusion.

\subsection{Hypersurfaces} The second instance we consider where one has additional information about the Hodge module $M$ and its filtered piece $F_{-w}\mathcal{M}$ is the case when $M$ is the minimal extension of the variation of Hodge structures arising from the vanishing cohomology of a family of hypersurfaces on a complex projective space $X$ determined by the choice of an ample line bundle $\mathcal{L}$ on $X$. Such Hodge modules $M$ were studied in detail by  Schnell in \cite{schnell2012residues}, who proved that such an $M$ has the property that $F_{-w}M$ is an ample vector bundle \cite[Theorem D (1)]{schnell2012residues}. We thus deduce in Section \ref{sec: schnell} that for such a Hodge module $M$ on a $d$-dimensional base $B$, the period integral map $\mathcal{P}_M$ induces, much as in the variation of Hodge structures case \eqref{eq: VHS pairing defn}, an honest \textit{period pairing}
\[K\otimes_\Q (\omega_B^{-1}\otimes_{\mathcal{O}_B}F_{-w} \cM)\rightarrow \cO_B[d].\]
The image of this pairing then consists precisely of the period integrals of the hypersurfaces in $X$. 

In \cite[Theorem 1.4]{HLZ}, Huang--Lian--Zhu prove, under the additional assumption that $X$ is a homogenenous space of a semisimple group, that the period integrals of the \textit{smooth} such hypersurfaces can be characterized as the solution space of a particular $\mathcal{D}$-module. Our morphism of perverse sheaves $\gamma_M\colon K\rightarrow L[-1]$ provides a characterization of the period integrals of \textit{all} such hypersurfaces. 
Under the Riemann--Hilbert correspondence between perverse sheaves and regular holonomic $\mathcal{D}$-modules, there should exist a $\mathcal{D}$-module whose solutions yield exactly this perverse sheaf $L[-1]$. It would be interesting to relate this $\mathcal{D}$-module to the one constructed by Huang--Lian--Zhu.

\subsection{Structure of the paper.} The structure of the paper is as follows. 

In Section \ref{sec: VHS}, we detail the definition of period integrals for variations of Hodge structures to make it most easily generalizable to the Hodge module setting. In doing so, we define the period integral map $\mathbb{P}_V$ for polarized variations of Hodge structures $V$. Then in Section \ref{sec: HM}, we detail our definition (Definition \ref{def: PI map HM}) of the period integral map $\mathcal{P}_M$ for polarized Hodge modules $M$. 

In Section \ref{sec: rest to VHS}, we examine the case that $M$ corresponds to a variation of Hodge structures $V$, proving in Proposition \ref{prop: VHS} that in this case $\mathcal{P}_M=\mathbb{P}_V[d]$. Section \ref{sec: min ext} then focuses on the case that $M$ is the minimal extension of a variation of Hodge structures $V$. Theorem \ref{prop: min ext} proves that, in this case, the morphism $\gamma_M$ obtained from the period integral map $\mathcal{P}_M$ is the minimal extension functor applied to the graph morphism of the variation of Hodge structure period integral map $\mathbb{P}_V$, shifted by $d$ the dimension of the base.

Sections \ref{sec: naturality} and \ref{sec: structure} develop some of the properties of the period integral map $\mathcal{P}_M$.  In particular, Proposition \ref{prop: inclusion} establishes the interaction between the period integral map for Hodge modules and polarized morphisms of Hodge modules. Proposition \ref{prop: M decomp} shows that $\mathcal{P}_M$ behaves well with respect to polarized direct sums. From these, Corollary \ref{cor: structure} relates $\mathcal{P}_M$ to Saito's Structure Theorem for pure Hodge modules. 

Lastly, Sections \ref{sec: 1 dim} and \ref{sec: schnell} detail the examples of one-dimensional families and families of hypersurfaces  respectively discussed above. In particular, Section \ref{sec: curves} examines  one-dimensional families of algebraic curves. Section \ref{sec: KL} examines one-dimensional families of projective varieties that have
mildly singular degenerations, with Proposition \ref{prop: kerr-laza} providing the application of Kerr--Laza's results in \cite{KL} to our setting.

\section{Period integrals for variations of Hodge structures}\label{sec: VHS}
Recall that a polarized variation of Hodge structures $V=(\mathcal{V}, \nabla, F^\bullet \mathcal{V}, \mathbb{V}, Q)$ of pure weight $w$ on a complex manifold $U$ of dimension $d$ consists of the data:
\begin{enumerate}
\item a holomorphic vector bundle $\mathcal{V}$ equipped with a flat connection $\nabla$,
\item a Hodge filtration $F^\bullet \mathcal{V}$ by holomorphic sub-bundles,
\item a local system $\mathbb{V}$ of finite-dimensional $\Q$-vector spaces,
\item a bilinear form $Q\colon \mathbb{V}\otimes_\Q \mathbb{V}\rightarrow \Q(-w)$, where $\Q(-w)=2\pi i \cdot \Q\subset \C$.
\end{enumerate}
The above are related by $\mathbb{V}\otimes_\Q \mathcal{O}_U=\mathcal{V}$ and $\mathbb{V}\otimes_\Q \C=\mathcal{V}^\nabla$, where $\Gamma(U,\mathcal{V}^\nabla)=\{v\in \Gamma(U,\mathcal{V})\mid \nabla v=0\}$. 

A variation of polarized Hodge structures $V=(\mathcal{V}, \nabla, F^\bullet \mathcal{V}, \mathbb{V}, Q)$  has the property that for every $u\in U$, the $\Q$-vector space $\mathbb{V}_u$ together with the filtration $F^\bullet \mathcal{V}_u$ define a $\Q$-Hodge structure polarized by the bilinear form $Q_u$. 

So let $V=(\mathcal{V}, \nabla, F^\bullet \mathcal{V}, \mathbb{V}, Q)$ be a variation of Hodge structures of pure weight $w$ on a complex manifold $U$ of dimension $d$. Note that we may view the polarization $Q$ as a map
\[Q\colon \mathbb{V}\rightarrow \mathbb{V}^\vee(-w),\]
where $\mathbb V^\vee$ denotes the local system dual to $\mathbb{V}$. Let 
$Q_{\mathcal{O}_U}\colon \mathcal{V}\rightarrow \cHom_{\mathcal{O}_U}(\mathcal{V}, \mathcal{O}_U)$ denote the extension of $Q$ to $\mathbb{V}\otimes_\Q \mathcal{O}_U=\mathcal{V}$. 

Now note that the embedding $\Q\hookrightarrow \C$ induces an inclusion
\begin{equation}\label{eq: s def}s\colon\mathbb{V}=\mathbb{V}\otimes_\Q \Q\hookrightarrow \mathbb{V}\otimes_\Q \C=\mathcal{V}^\nabla.\end{equation}

Additionally consider the canonical inclusions
\begin{equation}\label{eq: t eq}t\colon \mathcal{V}^\nabla \rightarrow \mathcal{V}\end{equation}
\begin{equation}\label{eq: i eq}i\colon F^w\mathcal{V}\hookrightarrow \mathcal{V}.\end{equation}

\begin{defn}\label{def: PI map VHS}
The \emph{period integral map} $\mathbb{P}_{V}$ for the polarized variation of Hodge structures $V$ is the composition
\[\mathbb{V}\xhookrightarrow{s}\mathcal{V}^\nabla\xhookrightarrow{t} \mathcal{V}\xrightarrow{Q_{\mathcal{O}_U}} \cHom_{\cO_U}(\mathcal{V},\mathcal{O}_U)\xrightarrow{i^\vee}\cHom_{\mathcal{O}_U}(F^w\mathcal{V},\mathcal{O}_U).\]
\end{defn}

We define the \emph{period integral pairing} for $V$ to be the induced pairing
\[\mathbb{V}\otimes_\Q F^w\mathcal{V}\rightarrow \mathcal{O}_U.\]

Observe that this period integral pairing for $V$ indeed recovers the pairing 
\begin{equation*}
V_\Q\otimes F^wV_\C\rightarrow \mathcal{O}_B, \hspace{.2in}
\langle \gamma, \omega\rangle \mapsto \int_{\gamma}\omega
\end{equation*}
introduced in \ref{eq: VHS pairing defn} in the introduction.


\section{Period integrals for Hodge modules}\label{sec: HM}

The goal of this section will be to define a period integral map for polarized Hodge modules analogous to the period integral map $\mathbb{P}_V$ for polarized varations of Hodge structures introduced in Section \ref{sec: VHS}. 

Recall that a polarized pure Hodge module $M=(\mathcal{M}, F_\bullet \mathcal{M}, K, \phi)$ of weight $w$ on a complex algebraic variety $B$ of dimension $d$ consists of the data:
\begin{enumerate}
\item a perverse sheaf $K$ over $\Q_B$,
\item a regular holonomic right $\mathcal{D}_X$-module $\mathcal{M}$ with an isomorphism
\[K\otimes_\Q \C \cong \DR(\mathcal{M}),\]
\item A good filtration $F_\bullet\mathcal{M}$ by $\mathcal{O}_B$-coherent subsheaves of $\mathcal{M}$ such that
\[F_p\mathcal{M}\cdot F_k\mathcal{D}\subset F_{p+k}\mathcal{M}\]
and such that the graded pieces $gr^F_\bullet \mathcal{M}$ are coherent over $gr^F_\bullet \mathcal{D}$,
\item A morphism $\phi\colon K(w)\rightarrow \mathbf{D}K$ extending to an isomorphism $ M(w)\rightarrow \mathbf{D}M$, where $\mathbf{D}K=R\cHom(K, \Q_B(d))[2d]$ denotes the Verdier dual of the perverse sheaf $K$.
\end{enumerate}
The properties that the above data need to satisfy in order for $M=(\mathcal{M}, F_\bullet \mathcal{M}, K, \phi)$ to actually be a polarized Hodge module can be found in \cite[Definition 12.5]{schnellnotes}.

We will begin by introducing a bit of notation. Let $\For$ denote the forgetful functor from the derived category of $\mathcal{D}_B$-modules to the derived category of $\mathcal{O}_B$-modules. 

Let us define another functor $H$ from the derived category of $\mathcal{O}_B$-modules to itself given by
\[
H(-) = R\cHom_{\cO_B}(\omega_B^{-1}\otimes_{\mathcal{O}_B}-, \cO_B).
\]
Therefore the composition $H\circ \For$ is a functor from the derived category of  $\mathcal{D}_B$-modules to the derived category of $\mathcal{O}_B$-modules.

Define the $\Sol$ functor for right $\mathcal{D}_B$-modules but with values in the left $\mathcal{D}_B$-module $\mathcal{O}_B$ by 
\[\Sol(-)=R\cHom_{\mathcal{D}_B}(\omega_B^{-1}\otimes_{\mathcal{O}_B}-,\mathcal{O}_B).\]
This functor $\Sol$ is a functor from the derived category of right $\mathcal{D}_B$-modules to the derived category of $\mathcal{O}_B$-modules.

\begin{lemma}\label{lem: eta lem}
There is a natural transformation $\eta$ from the functor $\Sol$ to $H\circ \For$ such that for any $\mathcal{D}_B$-module $\mathcal{N}$ the map 
\[\eta_\mathcal{N}\colon R\cHom_{\mathcal{D}_B}(\omega_B^{-1}\otimes_{\mathcal{O}_B}\mathcal{N},\mathcal{O}_B)\rightarrow R\cHom_{\cO_B}(\For(\omega_B^{-1}\otimes_{\mathcal{O}_B}\mathcal{N}),\cO_B)\] is a natural transformation.
\end{lemma}

\begin{proof}
For a $\cD_B$-module $\cN$ we define $\eta_\cN$ as follows.  Let $(\cA_{\bullet},d_{\bullet})$ be a projective resolution of $\omega_B^{-1} \otimes_{\cO_B} \cN$.  Since $\cD_B$ is free over $\cO_B$, this is also a projective resolution of $\omega_B^{-1} \otimes_{\cO_B} \cN$ as an $\cO_B$-module. Then define $\eta_\cN$ in degree $i$
\[
\eta_i \colon \cHom_{\cD_B}(\cA_i,\cO_B) \to \cHom_{\cO_B}(\For(\cA_i),\cO_B)
\]
to be the natural inclusion of the set of $\cD_B$-linear maps into the set of $\cO_B$-linear maps. 

The differentials $\cHom_{\cD_B}(d_i,\cO_B) = d_i^*$ and $\cHom_{\cO_B}(d_i,\cO_B) = d_i^*$ are both given by precomposition with $d_i$.  Thus for $\alpha$ a local section of $\cHom_{\cD_B}(\cA_i,\cO_B)$,
\begin{equation}\label{eqn:etachainmap}
   (\eta_{i+1} \circ d_{i+1}^*) (\alpha) = \alpha \circ d_{i+1} = (d_{i+1}^* \circ \eta_{i})  (\alpha),
\end{equation}
which shows that $\eta_\cN$ is a chain map. A similar equation shows that $\eta_\cN$ is well-defined up to quasi-isomorphism.

We now show $\eta$ is natural. Let $f\colon \cN_1 \to \cN_2$ be a morphism of $\cD_B$-modules.  Let $\cA_\bullet$ and $\cB_\bullet$ be projective resolutions of of $\cN_1$ and $\cN_2$ respectively.  Denote the map induced by $f$ on the resolutions by $f_i \colon \cA_i \to \cB_i$.  Similar to before, both maps 
$\cHom_{\cD_B}(f_i,\cO_B) = f_i^*$ and $\cHom_{\cO_B}(f_i,\cO_B) = f_i^*$ are given by precomposition and thus
\[
\begin{tikzcd}
 \cHom_{\cD_B}(\cA_i,\cO_B) \arrow[d, hook, "\eta_{\cN_1}"] & \cHom_{\cD_B}(\cB_i,\cO_B) \arrow[d, hook, "\eta_{\cN_2}"] \arrow[l, swap, "f_i^*"]\\
\cHom_{\cO_B}(\cA_i,\cO_B) & \cHom_{\cO_B}(\cB_i,\cO_B) \arrow[l, swap, "f_i^*"].
\end{tikzcd}
\]
is commutative.  Hence $\eta$ is natural.

\end{proof}

In what follows, we will abuse notation and omit the notation $\For$ for the forgetful functor, using  $\mathcal{N}$ to denote both the $\mathcal{D}_B$-module and the underlying $\mathcal{O}_B$-module.

The de Rham functor from the derived category of right $\mathcal{D}_B$-modules to the perverse sheaves is defined by
\[
\DR(\bullet)=\bullet \otimes_{\mathcal{D}_B}^L \cO_B.
\]
Let $\Theta_B$ be the tangent bundle of $B$.  The Spencer complex $\mathrm{Sp}^i = \mathcal{D}_B \otimes_{\cO_B} \wedge^{-i} \Theta_B$ is a resolution of $\mathcal{O}_B$ as a left $\mathcal{D}_B$-module in degrees $-d \leq i \leq 0$.  Using this we can write $\DR(\cM)$ for a right $\mathcal{D}_B$-module $\cM$ as 
\[
\DR ( \cM ) = \left[ \cM \otimes_{\cO_B} \wedge^{n} \Theta_B \to \ldots \to \cM \otimes_{\cO_B} \wedge^{1} \Theta_B \to \cM \right]
\]
where the rightmost term is in degree $0$.

Consider moreover the dualizing functor on right $\mathcal{D}_B$-modules
\[\mathbb{D}(\mathcal{M})=R\cHom_{\mathcal{D}_B}(\mathcal{M}, \omega_B\otimes_{\mathcal{O}_B}\mathcal{D}_B)[d].\]

In the case of a filtered regular holonomic $\mathcal{D}_B$-module $\mathcal{M}$, the Spencer complex $\DR(\mathcal{M})$ is a perverse sheaf. In this case, there is then an isomorphism \cite[Propsition 4.2.1]{HTT}
\begin{equation}\label{eq: dr sol}\DR(\mathcal{M})\cong \Sol(\mathbb{D}\mathcal{M})[d]\end{equation}
as well as isomorphisms \cite[Proposition 4.7.9]{HTT}
\begin{equation}\label{eq: dr sol 2}
\Sol(\mathbb{D}\mathcal{M})[d]\cong \mathbf{D}(\Sol(\mathcal{M})[d]),
\end{equation}
\begin{equation}\label{eq: dr dual}
\DR(\mathbb{D}\mathcal{M})\cong \mathbf{D}(\DR(\mathcal{M}))
\end{equation}
where $\mathbf{D}(-)$ denotes Verdier duality. Combining \eqref{eq: dr sol} and \eqref{eq: dr sol 2} and taking the Verdier dual of both sides yields
\[\mathbf{D}(\DR(\mathcal{M}))\cong \Sol(\mathcal{M})[d].\]
Combining this with \eqref{eq: dr dual} yields an isomorphism
\begin{equation}\label{eq: dr dual sol}
\DR(\mathbb{D}\mathcal{M})\cong \Sol(\mathcal{M})[d].
\end{equation}

Let $M=(\mathcal{M}, F_\bullet \mathcal{M}, K, \phi)$ be a pure Hodge module of weight $w$ with polarization $\phi$ on a complex algebraic variety $B$ of dimension $d$.  Recall that $\phi$ induces an isomorphism
\[ M(w)\xrightarrow{\sim} \mathbb{D}M.\]
In particular, we have a morphism
\[\varphi\colon \mathcal{M}(w)\rightarrow \mathbb{D}\mathcal{M}\]
with the property $\DR(\varphi)=\phi_\C$, where $\phi_\C$ denotes the extension of the polarization $\phi$ to $K\otimes_\Q\C=\DR(\mathcal{M})$. 

Applying the isomorphism \eqref{eq: dr dual sol} to $\DR(\varphi)=\phi_\C$, we obtain
\[\phi_\C\colon \colon \DR(\mathcal{M})\rightarrow \Sol(\mathcal{M})[d].\]

Then letting $\iota\colon F_{-w}\mathcal{M}\hookrightarrow \mathcal{M}$ denote the natural inclusion of $\mathcal{O}_B$-modules and $\sigma\colon K\hookrightarrow \DR(\mathcal{M})$ the inclusion induced by the canonical embedding $\Q\hookrightarrow \C$, we define
\begin{defn}\label{def: PI map HM}
The \emph{period integral map} $\mathcal{P}_M$ for a polarized Hodge module $M$ is the composition 
\begin{equation}\label{eq: P comp}
\scriptstyle K\xhookrightarrow{\sigma}  \DR(\mathcal{M})\xrightarrow{\phi_\C}  \Sol(\mathcal{M})[d]\xhookrightarrow{\eta_\mathcal{M}[d]}  R\cHom_{\cO_B}(\omega_B^{-1}\otimes_{\mathcal{O}_B}\mathcal{M},\cO_B)[d]\xrightarrow{H(\iota)[d]}R\cHom_{\cO_B}(\omega_B^{-1}\otimes_{\mathcal{O}_B}F_{-w}\mathcal{M},\cO_B)[d].\end{equation}
\end{defn}


\section{The map \texorpdfstring{$\mathcal{P}_M$}{PM} in the VHS case}\label{sec: rest to VHS}

For any polarized variation of Hodge structures $V=(\mathcal{V}, \nabla, F^\bullet \mathcal{V}, \mathbb{V}, Q)$ of weight $w$ on a complex algebraic variety of dimension $d$, one can associate a 
Hodge module $M_V=(\mathcal{M}, F_\bullet \mathcal{M}, K, \phi)$ of weight $w+d$  by defining
\begin{enumerate}
\item the perverse sheaf $K=\mathbb{V}[d]$,
\item the $\mathcal{D}_B$-module $\mathcal{M}=\omega_B\otimes_{\mathcal{O}_B} \mathcal{V}$,
\item the filtration $F_p \mathcal{M}=\omega_B\otimes_{\mathcal{O}_B}  F^{-p-d}\mathcal{V}$,
\item the polarization $\phi\colon K(w+d)\rightarrow \mathbf{D}K$ induced by $Q\colon \mathbb{V}(w)\rightarrow \mathbb{V}^\vee$.
\end{enumerate}

Hence if $M=M_V$ is a polarized Hodge module attached to some polarized variation of Hodge structures $V$, then there is a period integral map $\mathcal{P}_M$ for the Hodge module $M$ as well as a period integral map $\mathbb{P}_V$ for the variation of Hodge structures $V$. 

\propVHS

\begin{proof}
Suppose that $M=M_V$ is a polarized Hodge module attached to some polarized variation of Hodge structures $V$. 
Note that in the case  $M=M_V$, we have $K\otimes_\Q\C=\mathbb{V}[d]\otimes_\Q\C=\mathcal{V}^\nabla[d]$. Hence we know $\DR(\mathcal{M})\cong \mathcal{V}^\nabla[d]$. It follows that the inclusion  $\sigma\colon K\hookrightarrow \DR(\mathcal{M})$ induced by the canonical embedding $\Q\hookrightarrow \C$ is in fact  
\begin{equation}\label{eq: s equation}
s[d]\colon \mathbb{V}[d] \hookrightarrow \mathcal{V}^\nabla[d],\end{equation}
where $s$ is the inclusion introduced in \eqref{eq: s def} in Section \ref{sec: VHS}.

Additionally, the polarization $\phi$ of $M$ is just $Q[d]$, where $Q$ is the polarization of $V$. Hence $\phi_\C$=$Q_\C[d]$, where $Q_\C[d]$ denotes the extension of $Q$ to $\mathbb{V}[d]\otimes_\Q\C\cong \mathcal{V}^\nabla[d]$. So the map $\phi_\C\colon \DR(\mathcal{M})\rightarrow \Sol(\mathcal{M})[d]$ becomes
\[Q_\C[d]\colon \mathcal{V}^\nabla[d]\rightarrow\Sol(\mathcal{M})[d].\]

Since $\mathcal{M}=\omega_B\otimes_{\mathcal{O}_B} \mathcal{V}$, we have
\[\Sol(\mathcal{M})=R\cHom_{\mathcal{D}_B}(\omega_B^{-1}\otimes_{\mathcal{O}_B}\omega_B\otimes_{\mathcal{O}_B} \mathcal{V},\mathcal{O}_B)=\cHom_{\mathcal{D}_B}(\mathcal{V},\mathcal{O}_B).\]

So, in fact, we may write the map $\phi_\C$ as
\begin{equation}\label{eq: phi equation}Q_\C[d]\colon \mathcal{V}^\nabla[d]\rightarrow\cHom_{\mathcal{D}_B}( \mathcal{V},\mathcal{O}_B)[d].\end{equation}

Similarly, using the definition of the natural transformation $\eta_\mathcal{M}$ from the proof of  Lemma \ref{lem: eta lem}, the substitution $\mathcal{V}=\omega_B^{-1}\otimes_{\mathcal{O}_B}\mathcal{M}$ yields that $\eta_\mathcal{M}$ is just the natural inclusion
\begin{equation}\label{eq: eta inclusion} \cHom_{\mathcal{D}_B}( \mathcal{V},\mathcal{O}_B)[d]\hookrightarrow \cHom_{\mathcal{O}_B}( \mathcal{V},\mathcal{O}_B)[d].\end{equation}

Similarly substituting $\mathcal{V}=\omega_B^{-1}\otimes_{\mathcal{O}_B}\mathcal{M}$ yields that the map
\[H(\iota)[d]\colon R\cHom_{\cO_B}(\omega_B^{-1}\otimes_{\mathcal{O}_B}\mathcal{M},\cO_B)[d]\xrightarrow{H(\iota)}R\cHom_{\cO_B}(\omega_B^{-1}\otimes_{\mathcal{O}_B}F_{-w-d}\mathcal{M},\cO_B)[d]\]
is just
\begin{equation}\label{eq: iota equation}i^\vee[d]\colon \cHom_{\cO_B}(\mathcal{V},\cO_B)[d]\rightarrow \cHom_{\cO_B}(F^w\mathcal{V},\cO_B)[d],\end{equation}
where $i\colon F^w\mathcal{V}\hookrightarrow \mathcal{V}$ is the natural inclusion \eqref{eq: i eq} in Section \ref{sec: VHS}.

Hence by combining \eqref{eq: s equation}, \eqref{eq: phi equation},\eqref{eq: eta inclusion},  and \eqref{eq: iota equation}, we obtain that in the case $M=M_V$, the  map $\mathcal{P}_{M}$ is given by the composition
\[\mathbb{V}[d]\xhookrightarrow{s[d]}\mathcal{V}^\nabla[d] \xrightarrow{Q_\C[d]} \cHom_{\mathcal{D}_B}(\mathcal{V},\mathcal{O}_B)[d]\xhookrightarrow{\eta_\mathcal{M}[d]} \cHom_{\cO_B}( \mathcal{V},\cO_B)[d]\xrightarrow{i^\vee[d]}\cHom_{\cO_B}(  F^{w}\mathcal{V},\cO_B)[d].\]

For comparison, as established in Definition \ref{def: PI map VHS}, the period integral map $\mathbb{P}_V$ with the shift $[d]$ is the composition
\[\mathbb{V}[d]\xhookrightarrow{s[d]}\mathcal{V}^\nabla[d]\xhookrightarrow{t[d]} \mathcal{V}[d]\xrightarrow{Q_{\mathcal{O}_B}[d]} \cHom_{\cO_B}(\mathcal{V},\mathcal{O}_B)[d]\xrightarrow{i^\vee[d]}\cHom_{\mathcal{O}_B}(F^w\mathcal{V},\mathcal{O}_B)[d].\]

Thus to show that $\mathcal{P}_M=\mathbb{P}_V[d]$ we just need to show that the following diagram commutes
\begin{equation}\label{dia: VHS1}
\begin{tikzcd}
\mathcal{V}^\nabla\arrow[r, "Q_\C"]\arrow[d, "t"] & \cHom_{\mathcal{D}_B}(\mathcal{V},\mathcal{O}_B)\arrow[d, "\eta_\mathcal{M}"]\\
\mathcal{V}\arrow[r, "Q_{\mathcal{O}_B}"]& \cHom_{\cO_B}(\mathcal{V},\mathcal{O}_B).
\end{tikzcd}
\end{equation}

 Observe that the $\mathcal{O}_B$-module $\cHom_{\cO_B}(\mathcal{V},\mathcal{O}_B)$ is a left $\mathcal{D}_B$-module, where the action of $\mathcal{D}_B$ on a local section $f$ of $  \cHom_{\cO_B}(\mathcal{V},\mathcal{O}_B)$  is given by 

\[(\partial f)(v)=\partial (f(v))-f(\partial v).\]
Thus let 
\[\cHom_{\cO_B}(\mathcal{V},\mathcal{O}_B)^\nabla=\{f\in \cHom_{\cO_B}(\mathcal{V},\mathcal{O}_B)\mid \partial f=0\}.\]
So then $f\in \cHom_{\cO_B}(\mathcal{V},\mathcal{O}_B)^\nabla$ if and only if $(\partial f)(v)=f(\partial v)$ for all $v\in \mathcal{V}$. In other words
\[\cHom_{\cO_B}(\mathcal{V},\mathcal{O}_B)^\nabla=\cHom_{\mathcal{D}_B}(\mathcal{V},\mathcal{O}_B).\]
So then Diagram \eqref{dia: VHS1} is just the diagram
\begin{equation}\label{dia: VHS2}
\begin{tikzcd}
\mathcal{V}^\nabla\arrow[r, "Q_\C"]\arrow[d, "t"] & \cHom_{\cO_B}(\mathcal{V},\mathcal{O}_B)^\nabla \arrow[d, "u"]\\
\mathcal{V}\arrow[r, "Q_{\mathcal{O}_B}"]& \cHom_{\cO_B}(\mathcal{V},\mathcal{O}_B),
\end{tikzcd}
\end{equation}
where $t$ and $u$ are the natural inclusions. Since $\varphi_{\mathcal{O}_U}$ is a $\mathcal{D}_B$-module homomorphism, it follows that it induces a homomorphism $\varphi^\nabla=Q_\C$ on flat sections. The commutativity of $\varphi^\nabla$ with the natural inclusions then follows. 
\end{proof}


\section{The map \texorpdfstring{$\mathcal{P}_M$}{PM} and minimal extension}\label{sec: min ext}

Suppose that $U\subset B$ is an open subset of a complex algebraic variety $B$ and consider the inclusion $j\colon U\hookrightarrow B$. Let $V$ be a polarized variation of Hodge structures on $U$. Then recall, as proven in \cite{saito}, that the \emph{minimal extension functor} $j_{!*}$ sends the polarized variation of Hodge structure $V$ to a uniquely defined polarized Hodge module $M$ on $B$ such that the restriction of $M$ to $U$ is the polarized Hodge module $M_V$. The $\mathcal{D}_B$-module $\mathcal{M}$ underlying this Hodge module $M$ is the minimal extension $j_{!*}\mathcal{V}$ of the vector bundle with connection $(\mathcal{V}, \nabla)$ underlying the variation of Hodge structures $V$. 

In order to relate the minimal extension functor for Hodge modules to the map $\mathcal{P}_M$, note that for any polarized Hodge module $M$ on $B$, the resulting map $\mathcal{P}_M$ is a morphism in the category $D^b(\Q_B)$, the derived category of the category of sheaves of $\Q$-vector spaces on $B$. This derived category $D^b(\Q_B)$ is triangulated, hence it follows that the morphism $\mathcal{P}_M$ yields a distinguished triangle in $D^b(\Q_B)$ given by
\begin{equation}\label{eq: dist triangle 1}
K\xrightarrow{\mathcal{P}_M}  R\cHom_{\cO_B}(\omega_B^{-1}\otimes_{\mathcal{O}_B}F_{-w}\mathcal{M},\cO_B)[d]
\rightarrow \cone{\mathcal{P}_M}\rightarrow K[1].
\end{equation}
Let $L$ be the homotopy image of $\mathcal{P}_M$ in $D^b(\Q_B)$, meaning that $L$ is the cone of the morphism \[R\cHom_{\cO_B}(\omega_B^{-1}\otimes_{\mathcal{O}_B}F_{-w}\mathcal{M},\cO_B)[d]
\rightarrow \cone{\mathcal{P}_M}.\]
Hence $L$ fits into  another distinguished triangle in $D^b(\Q_B)$ given by
\begin{equation}\label{eq: dist triangle 2}
R\cHom_{\cO_B}(\omega_B^{-1}\otimes_{\mathcal{O}_B}F_{-w}\mathcal{M},\cO_B)[d]
\rightarrow \cone{\mathcal{P}_M}\rightarrow L\rightarrow R\cHom_{\cO_B}(\omega_B^{-1}\otimes_{\mathcal{O}_B}F_{-w}\mathcal{M},\cO_B)[d+1].
\end{equation}
From the identifications of the terms in the distinguished triangles \eqref{eq: dist triangle 1} and \eqref{eq: dist triangle 2}, it follows that the morphism $\mathcal{P}_M$ factors through an isomorphism 
\[\gamma_M\colon K\rightarrow L[-1].\]
In particular, the object $L$ of $D^b(\Q_B)$ is a perverse sheaf. Since the category of perverse sheaves $\Perv(\Q_B)$ is a full subcategory of $D^b(\Q_B)$ it follows that $\gamma_M$ is a morphism in the category $\Perv(\Q_B)$.

\thmMinExt

\begin{proof}
Let us write  $V=(\mathcal{V}, \nabla, F^\bullet \mathcal{V}, \mathbb{V}, Q)$ and $M=(\mathcal{M}, F_\bullet \mathcal{M}, K, \phi)$. Recall that the Hodge module $M_V$ associated to $V$ is given by
\[ M_V=(\omega_B\otimes_{\mathcal{O}_B} \mathcal{V}, F_p \mathcal{M_V}=\omega_B\otimes_{\mathcal{O}_B}  F^{-p-d}\mathcal{V}, \mathbb{V}[d], Q).\]
Then since 
$M=j_{!*}V$, we know in particular that $K=j_{!*}(\mathbb{V}[d]).$
Hence both $\gamma_M$ and $j_{!*}(\Gamma_{\mathbb{P}_V[d]})$ are morphisms of perverse sheaves with domain the perverse sheaf $K$.

Note that the minimal extension to $B$ of a simple local system $\mathbb{L}$ on $U\subset B$ is characterized by the property that it is the unique simple perverse sheaf on $B$ restricting to $\mathbb{L}$ on $U$ (\cite{beilinson2018faisceaux}). Since $\mathbb{V}$ is simple as a $\Q$-local system, it follows that $K$ is simple as a perverse sheaf. Hence both the perverse sheaf $L$ and the image of $j_{!*}(\Gamma_{\mathbb{P}_V[d]})$ are simple as perverse sheaves.

Now let us consider the restriction of the morphism $\gamma_M$ to the open set $U\subset B$. It follows from Proposition \ref{prop: VHS} that restricted to $U$, the morphism $\mathcal{P}_M$ is just the morphism of $\Q$-local systems 
\[\mathbb{P}_V[d]\colon \mathbb{V}[d]\rightarrow \cHom_{\cO_U}(\omega_U^{-1}\otimes_{\mathcal{O}_U}F^w\mathcal{V},\cO_U)[d].\] 
Hence after restricting to $U$, the distinguished triangle \eqref{eq: dist triangle 1} is the complex of $\Q$-local systems
\begin{equation}\label{diagram: cone1}
\adjustbox{scale=0.8,center}{%
\begin{tikzcd}
0\arrow[d] & 0\arrow[d] & 0\arrow[d] & 0\arrow[d] & \\
0\arrow[r]\arrow[d] & 0\arrow[r]\arrow[d] & \mathbb{V}\arrow[d, "\mathbb{P}_V"]\arrow[r,"\mathrm{Id}"] & \mathbb{V} \arrow[d] & \mathrm{deg} = -d-1\\
\mathbb{V}\arrow[r,"\mathbb{P}_{V}"]\arrow[d]& \cHom_{\mathcal{O}_U}(F^w\mathcal{V},\mathcal{O}_U)\arrow[r, "\mathrm{Id}"]\arrow[d]& \cHom_{\mathcal{O}_U}(F^w\mathcal{V},\mathcal{O}_U)\arrow[r]\arrow[d]&0\arrow[d] & \!\!\!\!\!\!\!\!\!\!\mathrm{deg} = -d\\
0&0&0&0.
\end{tikzcd}
}
\end{equation}
Similarly, after restricting to $U$, the distinguished triangle \eqref{eq: dist triangle 2} is the complex 
\begin{equation}\label{diagram: cone2}
\adjustbox{scale=0.85,center}{%
\begin{tikzcd}
0\arrow[d]&0\arrow[d]&0\arrow[d]&0\arrow[d]\\
0\arrow[r]\arrow[d]&\mathbb{V}\arrow[r, hookrightarrow]\arrow[d, "\mathbb{P}_V"]&\mathbb{V}\oplus \cHom_{\mathcal{O}_U}(F^w\mathcal{V},\mathcal{O}_U)\arrow[d]\arrow[r, twoheadrightarrow]& \cHom_{\mathcal{O}_U}(F^w\mathcal{V},\mathcal{O}_U)\arrow[d]\\
\cHom_{\mathcal{O}_U}(F^w\mathcal{V},\mathcal{O}_U)\arrow[r,"\mathrm{Id}"]\arrow[d]& \cHom_{\mathcal{O}_U}(F^w\mathcal{V},\mathcal{O}_U)\arrow[r, "\mathrm{Id}"]\arrow[d]& \cHom_{\mathcal{O}_U}(F^w\mathcal{V},\mathcal{O}_U)\arrow[r]\arrow[d]&0\arrow[d]\\
0&0&0&0,
\end{tikzcd}
}
\end{equation}
also in degrees $-d-1$ and $-d$.

A diagram chase then reveals that the morphism $\gamma_M|_U\colon K|_U\rightarrow L[-1]|_U$  is the morphism of complexes of $\Q$-local systems which is $0$ in every degree except degree $-d$, where it is the graph morphism
$\Gamma_{\mathbb{P}_V[d]}$. 
Indeed, from diagram \eqref{diagram: cone1}, the morphism  
\[\cone(\mathbb{P}_V[d])[-1]\rightarrow K\] is the identity on $\mathbb{V}$ in degree $-d$ and is $0$ in every other degree.
From Diagram \eqref{diagram: cone2}, the morphism
\[\cone(\mathbb{P}_V[d])[-1]\rightarrow L[-1]\]
is just the inclusion 
$\mathbb{V}\hookrightarrow \mathbb{V}\oplus \cHom_{\mathcal{O}_U}(F^w\mathcal{V},\mathcal{O}_U)$ in degree $-d$,  is the identity on $\cHom_{\mathcal{O}_U}(F^w\mathcal{V},\mathcal{O}_U)$ in degree $-d+1$, and is $0$ in all other degrees.  
From the fact that $\gamma_M$ fits into the commutative square
\[
\begin{tikzcd}
\cone(\mathbb{P}_V[d])[-1]\arrow[r]\arrow[d, equal] & K\arrow[d, "\gamma_M"]\\
\cone(\mathbb{P}_V[d])[-1]\arrow[r]& L[-1],
\end{tikzcd}
\]
it follows $\gamma_M|_U$ is zero in all degrees except $-d$ and in degree $-d$ is given by $v\mapsto (v, \mathbb{P}_V(v))$ for all  $v\in \mathbb{V}$.

Hence we have established that $\gamma_M$ is a morphism of simple perverse sheaves whose restriction to $U$ is the morphism $\Gamma_{\mathbb{P}_V[d]}$, which is a morphism of simple local systems. It follows that $\gamma_M=j_{!*}(\Gamma_{\mathbb{P}_V[d]})$.

\end{proof}


\section{Naturality of \texorpdfstring{$\mathcal{P}_M$}{PM} for polarized morphism}\label{sec: naturality}

Recall that we used the notation $H$ for the functor on the derived category of $\mathcal{O}_B$-modules
\[
H(-) = R\cHom_{\cO_B}(\omega_B^{-1}\otimes_{\mathcal{O}_B}-,\cO_B).
\]
We will denote by $HF_{-w}$ the functor on the derived category of $\mathcal{O}_B$-modules given by
\[
HF_{-w}(-) = R\cHom_{\cO_B}(\omega_B^{-1}\otimes_{\mathcal{O}_B}F_{-w}-,\cO_B).
\]
Recall from the definition of $\mathcal{P}_M$, that for a Hodge module $M=(\mathcal{M}, F_\bullet \mathcal{M}, K)$ we use the notation $\iota$ for the natural inclusion of $\mathcal{O}_B$-modules $\iota\colon F_{-w}\mathcal{M}\hookrightarrow \mathcal{M}$.

\begin{prop}\label{prop: inclusion}
If $f\colon M_1\rightarrow M_2$ is a morphism of polarized Hodge modules on a complex manifold $B$, then the following diagram commutes
\begin{equation}\label{dia: incl}
\begin{tikzcd}
K_1\arrow[r, "f_K"]\arrow[d, "\mathcal{P}_{M_1}"] & K_2 \arrow[d, "\mathcal{P}_{M_2}"]\\
R\cHom_{\cO_B}(\omega_B^{-1}\otimes_{\mathcal{O}_B}F_{-w}\mathcal{M}_1,\cO_B)[d]& R\cHom_{\cO_B}(\omega_B^{-1}\otimes_{\mathcal{O}_B}F_{-w}\mathcal{M}_2,\cO_B)[d]\arrow[l, "HF_{-w}(\iota)"].
\end{tikzcd}
\end{equation}

\end{prop}

\begin{proof}
Write $M_1=(\mathcal{M}_1, F_\bullet \mathcal{M}_1, K_1, \phi_1)$ and $M_2=(\mathcal{M}_2, F_\bullet \mathcal{M}_2, K_2, \phi_2)$. Since $f$ is a morphism of polarized Hodge modules, we have the equality
\[\phi_1=\mathbb{D}(f)\circ \phi_2\circ f\]
and the following diagram commutes
\[
\begin{tikzcd}
K_1 \arrow[r," \sigma_1"] \arrow[d,"f_K"]  &\DR(\cM_1) \arrow[r," \phi_{1,\C}"] \arrow[d,"\DR(f_\mathcal{M})"]  &   R\mathcal{H}om_{\mathcal{D}}(\omega_B^{-1}\otimes_{\mathcal{O}_B}\cM_1,\cO_B)[d]    \\
K_2\arrow[r,"\sigma_2"]   & \DR(\cM_2) \arrow[r," \phi_{2,\C}"]   &   \arrow[u,"\DR(\mathbb{D}(f_\mathcal{M}))"]R\mathcal{H}om_{\mathcal{D}}(\omega_B^{-1}\otimes_{\mathcal{O}_B}\cM_2,\cO_B)[d].
\end{tikzcd}
\]

Moreover by Lemma \ref{lem: eta lem}, we know $\eta$ is a 
 a natural transformation from the functor $\Sol$ to $H\circ F$, so the following square commutes

\[
\begin{tikzcd}
   R\mathcal{H}om_{\mathcal{D}}(\omega_B^{-1}\otimes_{\mathcal{O}_B}\cM_1,\cO)[d] \arrow[r,"\eta_{\cM_1}"] & R\mathcal{H}om_{\cO_B}(\omega_B^{-1}\otimes_{\mathcal{O}_B}\cM_1,\cO_B)[d] \\
  R\mathcal{H}om_{\mathcal{D}}(\omega_B^{-1}\otimes_{\mathcal{O}_B}\cM_2,\cO)[d] \arrow[r,"\eta_{\cM_2}"]  \arrow[u,"\DR(\mathbb{D}(f_M))"] &  R\mathcal{H}om_{\cO_B}(\omega_B^{-1}\otimes_{\mathcal{O}_B}\cM_2,\cO)[d] \arrow[u,"R\mathcal{H}om_{\cO_B} (f _M{,} \cO_B) "] .
\end{tikzcd}
\]

Lastly the morphism $f$ must behave compatibly with the filtrations on $\mathcal{M}_1$ and $\mathcal{M}_2$. In particular, we have $f_{\mathcal{M}}(F_{-w}\mathcal{M}_1)\subset F_{-w}\mathcal{M}_2$. Combining this fact with the above two commutative diagrams yields the commutativity of the following diagram

\[
\adjustbox{scale=0.8,center}{%
\begin{tikzcd}
K_1 \arrow[r," \sigma_1"] \arrow[d,"f_K"]  &
\DR(\cM_1) \arrow[r," \phi_{1,\C}"] \arrow[d,"\DR(f_\mathcal{M})"]  &
R\mathcal{H}om_{\mathcal{D}_B}(\omega_B^{-1}\otimes_{\mathcal{O}_B}\cM_1,\cO_B)[d]  \arrow[r,"\eta_{\cM_1}"] &
R\mathcal{H}om_{\cO_B}(\omega_B^{-1}\otimes_{\mathcal{O}_B}\cM_1,\cO_B)[d]  \arrow[r,"H(i_{\cM_1})"] &
R\mathcal{H}om_{\cO_B}(\omega_B^{-1}\otimes_{\mathcal{O}_B}F_{-w} \cM,\cO_B)[d] 
\\
K_2 \arrow[r," \sigma_2"]   &
\DR(\cM_2) \arrow[r," \phi_{2,\C}"]   &
R\mathcal{H}om_{\mathcal{D}_B}(\omega_B^{-1}\otimes_{\mathcal{O}_B}\cM_2,\cO_B)[d] \arrow[r,"\eta_{\cM_2}"]  \arrow[u,"\DR(\mathbb{D}(f_\mathcal{M}))"]&
 R\mathcal{H}om_{\cO_B}(\omega_B^{-1}\otimes_{\mathcal{O}_B}\cM_2,\cO_B)[d] \arrow[r,"H(i_{\cM_2})"] \arrow[u,"R\mathcal{H}om_{\cO_B}(f_\mathcal{M} {,} \cO_B) "] &
R\mathcal{H}om_{\cO_B}(\omega_B^{-1}\otimes_{\mathcal{O}_B}F_{-w} \cM_2,\cO_B)[d] \arrow[u,"R\mathcal{H}om_{\cO_B}(f_{\mathcal{M}}\circ \iota {,} \cO_B) "]
.
\end{tikzcd}
}
\]
The top row of the diagram is $\mathcal{P}_{M_1}$ and the bottom row is $\mathcal{P}_{M_2}$. Hence we indeed have 
\[\mathcal{P}_{M_1}=R\mathcal{H}om_{\cO_B}(f_{\mathcal{M}}\circ \iota {,} \cO_B)\circ \mathcal{P}_{M_2}\circ f_K.\]
\end{proof}

Recall that for any polarized Hodge module $M=(\mathcal{M}, F_\bullet \mathcal{M}, K, \phi)$, the period integral map $\mathcal{P}_M$ factors through the isomorphism of perverse sheaves  
$\gamma_M\colon K\rightarrow L[-1]$ defined in Section \ref{sec: min ext}. 
\begin{cor}\label{cor: gamma incl}
If $f\colon M_1\rightarrow M_2$ is a morphism of polarized Hodge modules on a complex manifold $B$, then the following diagram commutes
\begin{equation}
\begin{tikzcd}
L_1\arrow[r, "\gamma_{M_2}\circ f_K\circ \gamma_{M_1}^{-1}"]\arrow[d, "\pi_{M_1}"] & L_2 \arrow[d, "\pi_{M_2}"]\\
R\cHom_{\cO_B}(\omega_B^{-1}\otimes_{\mathcal{O}_B}F_{-w}\mathcal{M}_1,\cO_B)[d+1]& R\cHom_{\cO_B}(\omega_B^{-1}\otimes_{\mathcal{O}_B}F_{-w}\mathcal{M}_2,\cO_B)[d+1]\arrow[l, "HF_{-w}(\iota)"],
\end{tikzcd}
\end{equation}
where $\pi_{M_1}$ and $\pi_{M_2}$ are maps satisfying $\pi_{M_1}\circ \gamma_{M_1}=\mathcal{P}_{M_1}$ and $\pi_{M_2}\circ \gamma_{M_2}=\mathcal{P}_{M_2}$.
\end{cor}

\begin{proof}
By Proposition \ref{prop: inclusion}, we know 
$\pi_{M_1}\circ \gamma_{M_1}=HF_{-w}(\iota)\circ \pi_{M_2}\circ \gamma_{M_2}\circ f_K.$
Precomposing both sides with the isomorphism $\gamma_{M_1}^{-1}$ then yields the result.
\end{proof}


\section{The structure of \texorpdfstring{$\mathcal{P}_M$}{PM}}\label{sec: structure}
The Structure Theorem for polarized Hodge modules, as proved in \cite{saito2},  describes how any polarized Hodge module $M$ can be decomposed as a direct sum of minimal extensions of variations of polarized Hodge structures. In order to relate this to the structure of the period integral map $\mathcal{P}_M$, we begin by observing that $\mathcal{P}_M$ behaves well under direct sums. 

\begin{prop}\label{prop: M decomp}
Let  $M\in HM(B,w)$ be a polarized  Hodge module on a complex algebraic variety $B$ such that $M$ decomposes as a polarized direct sum
$M=\bigoplus_i M_i$. Then the period integral map for $M$ decomposes as
\[\mathcal{P}_M=\bigoplus_i \mathcal{P}_{M_i}.\]
\end{prop}

\begin{proof}
Note that the decomposition $M=\bigoplus_i M_{i}$ induces a decomposition
\begin{equation}\label{eq:RHomsplits}
\begin{aligned}
R\mathcal{H}om_{\cO_B}(\omega_B^{-1}\otimes_{\mathcal{O}_B}F_{-w} ( \bigoplus_i \mathcal{M}_{i} ),\cO_B)[d]&\cong 
R\mathcal{H}om_{\cO_B}(\omega_B^{-1}\otimes_{\mathcal{O}_B} \bigoplus_i F_{-w}\mathcal{M}_{i}),\cO_B)[d]\\
&\cong R\mathcal{H}om_{\cO_B}( \bigoplus_i (\omega_B^{-1}\otimes_{\mathcal{O}_B}F_{-w}\mathcal{M}_{i}),\cO_B)[d]\\
&\cong \bigoplus_iR\mathcal{H}om_{\cO_B}(\omega_B^{-1}\otimes_{\mathcal{O}_B}F_{-w} \mathcal{M}_{i},\cO_B)[d].
\end{aligned}
\end{equation}
In particular, it follows from Proposition \ref{prop: inclusion} that all of the  $\mathcal{P}_{M_i}$ may be computed via the inclusions $M_i\hookrightarrow M$. 

In fact, one may verify using \eqref{eq:RHomsplits} that under the assumption that $\phi_M=\bigoplus_i\phi_{M_i}$ all of the maps in \eqref{eq: P comp} which form the composition $\mathcal{P}_M$ are compatible with the decomposition $M=\bigoplus_i M_{i}$.  That is $\mathcal{P}_M|_{K_i}$ maps to $R\mathcal{H}om_{\cO_B}(\omega_B^{-1}\otimes_{\mathcal{O}_B}F_{-w} \mathcal{M}_{i},\cO_B)[d]$ and coincides with $\mathcal{P}_{M_i}$. Since $\mathcal{P}_M$ is a morphism in the derived category $D^b(\Q_B)$, which has direct sums, it follows that  $\mathcal{P}_M=\bigoplus_i \mathcal{P}_{M_{i}}$. 
\end{proof}

We obtain the following immediate consequence  of Proposition \ref{prop: M decomp} together with Corollary \ref{cor: gamma incl}.

\begin{cor}\label{cor: gamma decomp}
Let  $M\in HM(B,w)$ be a polarized  Hodge module on a complex algebraic variety $B$ such that $M$ decomposes as a polarized direct sum
$M=\bigoplus_i M_i$. 
Letting $\gamma_{M_i}\colon K_i\rightarrow L_i[-1]$ be isomorphisms constructed as in Section  \ref{sec: min ext}, then $\mathcal{P}_M$ factors through the isomorphism of perverse sheaves $\gamma_M=\bigoplus_i \gamma_{M_i}$.
\end{cor}

Saito's Structure Theorem  for polarized Hodge modules \cite{saito2} asserts that if $B$ is a complex algebraic variety, then every Hodge module $M\in HM(B,w)$ decomposes as a direct sum
\[M=\bigoplus_{Z\subset B}M_Z,\]
where $M_Z$ is a polarizable Hodge module of weight $w-\dim Z$ with strict support $Z$ obtained as $M_Z=j_{!*}V_Z$, where $V_Z$ is a polarized variation of Hodge structures on some open subset $j\colon U\hookrightarrow Z$ contained in the smooth locus of $Z$.

In particular, there exists a stratification $B=\bigsqcup B_i$ of $B$ into locally closed subvarieties such that 
\begin{equation}\label{eq: M decomp}M=\bigoplus_i M_{i},\end{equation}
where each $M_i=(j_{i})_{!*}V_{B_i}$ for $V_{B_i}$ a simple variation of Hodge structures on $B_i$ and $j_i\colon B_i\hookrightarrow \overline{B}_i$. 

\corStructure

\begin{proof}
This immediately follows from Theorem \ref{prop: min ext} applied to Corollary \ref{cor: gamma decomp}.
\end{proof}



\section{Example: The case when \texorpdfstring{$\dim B=1$}{dim B = 1}}\label{sec: 1 dim}

In this section, we focus our attention on the case of 
 a polarized Hodge module $M=(\mathcal{M}, F_\bullet \mathcal{M}, K, \phi)$ over a $1$-dimensional complex algebraic variety $B$.
We consider the Hodge module $M$ locally on a disk $\Delta\subset B$. Let $f\colon \Delta \rightarrow \mathbb{C}$ be a non-constant holomorphic function such that $f^{-1}(0)=\{0\}\in \Delta$. 

Recall that one may then apply the nearby cycles functor $\Psi_f$ for $f$ to the Hodge module $M$ to obtain a mixed Hodge module $\Psi_fM$ supported at the point $\{0\}$ whose underlying perverse sheaf and $\mathcal{D}$-module arise as the nearby cycles functor applied to the perverse sheaf and $\mathcal{D}$-module underlying the Hodge module $M$. See \cite[Section 8]{schnellnotes}, for instance, for more on this construction.

Let $T$ be the monodromy operator which acts on $\Psi_fM$. Then  the logarithm of the unipotent part of the monodromy $N_u=\frac{1}{2\pi i}T_u$ yields a a nilpotent endomorphism (up to Tate twist) of $\Psi_fM$. Hence $N_u$ induces a weight filtration $W(N_u)$ of $\Psi_fM$ such that for any $\ell \in \Z$, the graded pieces $\gr^{W(N_u)}_{\ell-w+1}\Psi_fM[-1]$ are pure Hodge modules supported at $\{0\}$.

\begin{prop}\label{prop: curve}
Suppose that $M$ is a polarized Hodge module  on a disk $\Delta\subset B$  such that $M$ is the minimal extension of a polarized variation of Hodge structures $V$ on $\Delta^*$. Then for every $\ell\in \Z$ and non-constant holomorphic function $f\colon \Delta\rightarrow \mathbb{C}$ such that $f^{-1}(0)=\{0\}$, 
\[\mathcal{P}_{\gr^{W(N_u)}_{\ell-w+1}\Psi_fM[-1]}=\mathcal{P}_{\gr^{W(N_u)}_{\ell-w+1}L_fV},\]
where $L_fV$ denotes the limit mixed Hodge structure in the sense of Deligne for the variation of Hodge structures $V$. 
\end{prop}

\begin{proof}
Note that since the graded pieces $\gr^{W(N_u)}_{\ell-w+1}\Psi_fM$ are all pure Hodge modules supported on a point, they must correspond to Hodge structures. Moreover, 
 by design, the filtrations on $\Psi_fM$ guarantee that the mixed Hodge structure corresponding to $\Psi_fM$ is exactly Deligne's limit mixed Hodge structure $L_fV$. Since the ${W(N_u)}$-graded pieces of $L_fV$ are polarized pure Hodge structures, the result then follows from Proposition \ref{prop: VHS}. 
\end{proof}

\subsection{Example: the case of a family of curves}\label{sec: curves}

We follow here the exposition of \cite[Section 5.4]{hain} to recall some additional information in the case of a family of curves. Suppose that $\mathscr{C}\rightarrow \Delta$ is a stable degeneration of a family of smooth algebraic curves of genus $g$ such that $\mathscr{C}$ is smooth and all fibers $C_t$ are smooth for $t\ne 0$. Then the central fiber $C_0$ is reduced and stable, meaning $C_0$ has finite automorphism group. 

Consider the polarized variation of Hodge structures $V$ on $\Delta^*$ given by $t\mapsto H^1(C_t,\Q)$ and, fixing a non-constant holomorphic function $f\colon \Delta\rightarrow \C$ such that $f^{-1}(0)=\{0\}$, let $L_fV$ denote the limit mixed Hodge structure at $\{0\}$. 
Then the monodromy weight filtration on $L_fV$ has length $3$ and 
\[F^0L_fV=L_fV, \hskip.2in F^2L_fV=0.\]
Hence 
\[\gr^{W(N_u)}_{\ell}L_fV\]
is a polarized Hodge structure of type $(0,0)$ when $\ell=0$, of type $(1,1)$ when $\ell=2$, and of type $(1,0), (0,1)$ when $\ell =1$. 

Moreover, if $\tilde{C}_0$ is the normalization of $C_0$ with decomposition into irreducible components $\tilde{C}_0=\bigcup_i X_i$, there is a natural isomorphism of polarized Hodge structures
\[\gr^{W(N_u)}_{1}L_fV\cong H^1(\tilde{C}_0)\cong \bigoplus_{i}H^1(X_i).\]

Hence, using Proposition \ref{prop: curve} together with Proposition \ref{prop: VHS}, the period integrals of the Hodge module $M$ given by the minimal extension of $V$ at the point $\{0\}$ are given by
\[\mathcal{P}_{\gr^{W(N_u)}_{1}\Psi_fM[-1]}=\bigoplus_i\mathbb{P}_{H^1(X_i)}.\]

\subsection{Example: mildly singular degenerations}\label{sec: KL}

In the case of a $1$-dimensional family of smooth projective varieties degenerating to a mildly singular variety, recent results of Kerr--Laza \cite{KL} yield the following. 

\begin{prop}\label{prop: kerr-laza}
Let $\sigma\colon \mathscr{X}\rightarrow \Delta$ be a flat projective family of $n$-dimensional varieties over the disk, which is the restriction of an algebraic family over a curve $B$, such that f is smooth over $\Delta^*$. Suppose that $\mathscr{X}$ is normal and $\Q$-Gorenstein and that the special fiber $X_0=\sigma^{-1}(0)$ is reduced and semi-log canonical (slc). Let $M$ be the polarized Hodge module on $B$ obtained locally on $\Delta$ as the minimal extension of the polarized variation of Hodge structures $V$ on $\Delta^*$ given by $\mathbb{V}_t=H^n(X_t, \Q)$.

If $i\colon \{0\}\rightarrow \Delta$ is the natural inclusion and $f\colon \Delta\rightarrow \C$ is a non-constant holomorphic function such that $f^{-1}(0)=\{0\}$, then for all $\ell \in \mathbb{Z}$ the periods of the special fiber $X_0$ may be recovered as
\[\mathcal{P}_{\gr^{W(N_u)}_{\ell-w+1}i^*M}=
\mathcal{P}_{\gr^{W(N_u)}_{\ell-w+1}\Psi_fM[-1]}.\]
\end{prop}

\begin{proof}
It follows from \cite[Theorem 1.2]{KL} that under the given hypotheses we have an isomorphism of Hodge structures
\[
\gr_{F}^{0} H^n(X_0,\Q)\cong\gr_{F}^{0}L_fV,
\]
where $L_fV$ denotes the limit mixed Hodge structure. Furthermore, using the compatibility between the weight and Hodge filtrations for the above mixed Hodge structures, we get that for any $\ell \in \mathbb{Z}$ the relation on both $H^n(X_0,\Q)$ and $L_fV$
\begin{equation}
    \begin{aligned}
 \gr^{W(N_u)}_\ell \gr_{F}^{0}&=\gr_{F}^{0}\gr^{W(N_u)}_\ell\\
 &\cong  F^{n}\gr^{W(N_u)}_\ell, 
 \end{aligned}
 \end{equation}
 where the second isomorphism is obtained via complex conjugation. 
 Namely, we have for any $\ell \in \mathbb{Z}$
 \[
F^{n}\gr^{W(N_u)}_\ell  H^n(X_0,\Q)\cong F^{n}\gr^{W(N_u)}_\ell L_fV.
\]
Passing to the corresponding polarized Hodge modules, we get
\begin{equation}\label{eq: filtered iso}
F_{-n}\gr^{W(N_u)}_\ell  i^*M\cong F_{-n}\gr^{W(N_u)}_\ell \Psi_fM,
\end{equation}
where recall that the Hodge module $M$ has weight $n+1$. 

Recall, moreover, that for any $\ell \in \mathbb{Z}$, there is a polarized inclusion of pure Hodge modules
\[
 \gr^{W(N_u)}_\ell i^*M[-1]\hookrightarrow \gr^{W(N_u)}_\ell\Psi_f M[-1] ,
\]
where the shift by $[-1]$ is included to ensure perversity of the underlying rational structure. 
Then Proposition \ref{prop: inclusion} applied to this inclusion together with the isomorphism \eqref{eq: filtered iso} implies  the result. 

\end{proof}

\section{Example: Hypersurfaces}\label{sec: schnell}
Let $X$ be a complex projective space of dimension $n$ and fix an ample line bundle $\mathcal{O}_X(1)$ on $X$ with the property that $H^q(X, \Omega_X^p(k))=0$ for every $p\ge 0$ and every $q,k>0$. Consider the family of hypersurfaces in $X$ determined by the $d$-dimensional the linear system $B=|\mathcal{O}_X(1)|$.  

Let $\mathscr{X}\subset X\times B$ denote the incidence variety and consider the projection map $\pi\colon \mathscr{X}\rightarrow B$. 
We may then consider the projection map 
$\pi^{\mathrm{sm}}\colon\mathscr{X}^{\mathrm{sm}}\rightarrow B^{\mathrm{sm}}$ 
given by restricting $\pi$ to the open subset $B^{\mathrm{sm}}$ over which $\pi$ is smooth.

Recall that if $D\subset X$ is a smooth hypersurface, the \emph{vanishing cohomology} is 
\[H^{n-1}_{\mathrm{ev}}(D,\Q)=\ker (i_*\colon H^{n-1}(D,\Q)\rightarrow H^{n+1}(X,\Q)),\]
where $i_*$ is the Gysin map for the inclusion $i\colon D\hookrightarrow X$. We thus consider the variation of Hodge structures on $B^{\mathrm{sm}}$ defined by the $\Q$-local system
\[\mathbb{V}=R^{n-1}_{\mathrm{ev}}\pi^{\mathrm{sm}}_*\Q=\ker(R^{n-1}_{\mathrm{ev}}\pi_*^{\mathrm{sm}}\Q\rightarrow H^{n+1}(X,\Q)),\]
whose fibers on points $b\in B^{\mathrm{sm}}$ are $H^{n-1}_{\mathrm{ev}}(\pi^{-1}(b),\Q)$.

Let $M=(\mathcal{M}, F_\bullet \mathcal{M}, K, \phi)$ be the polarized Hodge module on $B$ which is the minimal extension of $\mathbb{V}$. Then it follows from \cite[Theorem D (1)]{schnell2012residues} that the filtered piece $F_{-w}\mathcal{M}$ of the $\mathcal{D}_B$-module $\mathcal{M}$ is an ample vector bundle. In particular, we have an identification by taking cohomology
\[R\mathcal{H}om_{\cO}(\omega_B^{-1}\otimes_{\mathcal{O}_B}F_{-w} \cM,\cO_B)[d] \cong \mathcal{H}om_{\cO}(\omega_B^{-1}\otimes_{\mathcal{O}_B}F_{-w} \cM,\cO_B)[d].\]
Hence the period integral map $\mathcal{P}_M$  is in fact just of the form
\[\mathcal{P}_M\colon K \rightarrow \mathcal{H}om_{\cO}(\omega_B^{-1}\otimes_{\mathcal{O}_B}F_{-w} \cM,\cO_B)[d].\]
Therefore  $\mathcal{P}_M$
induces, as in the Hodge structure case, a \emph{period pairing}
\[K\otimes_\Q (\omega_B^{-1}\otimes_{\mathcal{O}_B}F_{-w} \cM)\rightarrow \cO_B[d]\]
whose image encodes the period integrals for the entire family of hypersurfaces determined by $|\mathcal{O}_X(1)|$, not just the smooth ones.

\textbf{Acknowledgements.} The authors would like to thank Yajnaseni Dutta, Bong Lian, Christian Schnell, and Lei Wu for helpful conversations. The first author was supported by the 
  National Science Foundation under Grant DMS-1803082 and Grant DMS-1645877 as well as under Grant DMS-1440140, while in residence at the Mathematical Sciences Research Institute in Berkeley, California during the Spring 2019 semester.  The second author was supported by the National Science Foundation under Grant DMS-1503050 and Grant DMS-1645877.

\bibliographystyle{amsalpha}
\bibliography{refs}

\end{document}